\documentclass[12pt, leqno]{amsart}
\usepackage{amscd,verbatim,amssymb}
\usepackage[mathscr]{eucal}
\usepackage{color}
\usepackage{verbatim}
\usepackage{graphics}
\usepackage{hyperref}
\usepackage{xy}
\xyoption{all}
\newtheorem{theorem}{Theorem}[section]
\newtheorem{proposition}[theorem]{Proposition}
\newtheorem{lemma}[theorem]{Lemma}
\newtheorem{corollary}[theorem]{Corollary}

\theoremstyle{definition}
\newtheorem{remark}[theorem]{Remark}

\newtheorem{examples}[theorem]{Examples}

\newcommand{\m}{^{\times}}

\newcommand{\inv}{^{-1}}

\newcommand{\Ad}{\operatorname{Ad}}

\newcommand{\ind}{\operatorname{ind}}

\newcommand{\Nrd}{\operatorname{Nrd}}
\newcommand{\Int}{\operatorname{Int}}
\newcommand{\ad}{\operatorname{ad}}

\newcommand{\ord}{\operatorname{ord}}
\newcommand{\gr}{\operatorname{\sf gr}}
\newcommand{\disc}{\operatorname{disc}}

\newcommand{\Z}{\mathbb{Z}}
\newcommand{\N}{\mathbb{N}}

\newcommand{\cA}{\mathcal A}
\newcommand{\cB}{\mathcal B}
\newcommand{\cC}{\mathcal C}

\numberwithin{equation}{section}
\title{Involutions on tensor products of quaternion algebras}
\date{}
\author[D. Barry]{Demba  Barry}
\address{DER de Math\'ematiques et Informatique, Universit\'e des Sciences des Techniques et des Technologies de Bamako, BP: E3206 Bamako, Mali}
\email{barry.demba@gmail.com}
\begin{document}
\maketitle
\begin{abstract}
We study possible decompositions of totally decomposable algebras with involution, that is, tensor products of quaternion algebras with involution. In particular, we are interested in  decompositions  in which one or several factors are the split quaternion algebra $M_2(F)$, endowed with an orthogonal involution. Using the theory of gauges, developed by Tignol-Wadsworth, we construct examples of algebras isomorphic to a tensor product of quaternion algebras with $k$ split factors, endowed with an involution which is totally decomposable, but does not admit any decomposition with $k$ factors $M_2(F)$ with involution. This extends an earlier result of Sivatski where the algebra considered is of degree $8$ and index $4$, and endowed with some orthogonal involution. 
\end{abstract}
\section{Introduction}\label{section1}
Let $F$ be a field of characteristic different from $2$. In ~\cite{Mer81}, Merkurjev showed that a central simple $F$-algebra of exponent $2$ is always  Brauer-equivalent to a tensor product of quaternion algebras. In degree $8$, this result  was first proved by Tignol (see for instance ~\cite[Thm. 5.6.38]{Jac96}). More precisely, let $A$ be a central simple $F$-algebra of degree $8$ and exponent $2$. Tignol showed that $M_2(A)$ is isomorphic to a tensor product of four quaternion algebras. Moreover, if $A$ is indecomposable (i.e, $A$ does not contain any quaternion subalgebra, see~\cite{ART79}), one may check that $M_2(A)$ does not admit any decomposition as a tensor product of  quaternion algebras in which one  the quarternion factors is split, that is isomorphic to $M_2(F)$. In this paper,  we are interested in the analogue question for  algebras with involution.

A central simple algebra $A$ of exponent $2$ is called  \emph{totally decomposable} if $A\simeq Q_1\otimes\cdots\otimes Q_m$, for some quaternion algebras $Q_i$.  The set  $\{Q_1,\ldots, Q_m\}$ is called a set of factors of $A$. Moreover, we say that $A$ is totally decomposable with $k$ split factors if it admits a set of factors containing $k$ split factors $M_2(F)$. This implies that the \emph{co-index} of $A$, which is the ratio $\deg A/\ind A$ is a multiple of $2^k$; but as explained before, this co-index condition is not sufficient for totally decomposable algebra to admit a decomposition with $k$ split factors. Assume now that $A$ is endowed with an $F$-linear involution $\sigma$. We say that $(A,\sigma)$ is totally decomposable if $(A, \sigma)\simeq (Q_1,\sigma_1)\otimes\cdots\otimes (Q_m,\sigma_m)$ and totally decomposable with $k$ split factors if it admits such a decomposition with $k$ factors isomorphic to $M_2(F)$ endowed with some $F$-linear involution. The main question we are interested in is the following: let $A$ be a central simple algebra of  co-index $2^k$ such that $A$ is totally decomposable with $k$ split factors. Let $\sigma$ be an anisotropic involution on $A$. \emph{If $(A,\sigma)$ is totally decomposable, is it totally decomposable with $k$ split factors}?

As we now proceed to explain,  a complete answer is known is the orthogonal case. Let $(A,\sigma)$ be a totally decomposable algebra with orthogonal involution of degree $2^m$.  If $A$ is split, Becher showed in ~\cite{Bec08} that $(A,\sigma)\simeq (M_{2^m}(F), \ad_\pi)$ where $\pi$ is an $m$-fold Pfister form over $F$ and $\ad_\pi$ is the involution  on $M_{2^m}(F)$ that is adjoint to the polar form of $\pi$. By definition of Pfister forms, it follows that  $(A,\sigma)$ is totally decomposable with $m$ split factors. This result is known in the literature as the Pfister Factor Conjecture (see ~\cite[Chap. 9]{Sha00}). Assume that $(A,\sigma)$ is of co-index $2^{m-1}$, i.e, the algebra $A$ is Brauer-equivalent to a quaternion algebra $H$. Again a positive answer is given  in ~\cite{Bec08}:  $(A,\sigma)\simeq (H,\sigma_H)\otimes (M_{2^{m-1}}(F), \ad_{\pi'})$ for some orthogonal involution $\sigma_H$ on $H$ and some $(m-1)$-fold Pfister form $\pi'$ over $F$. In higher index, a counter example is provided by Sivatski in ~\cite[Prop. 5]{Siv05}. He produces an algebra with orthogonal involution of degree $8$ and co-index $2$ which is totally decomposable, but not totally decomposable with a split factor. Observe that the algebra in this example is isomorphic to $M_2(D)$ for some biquaternion algebra $D$. Hence, it is totally decomposable with a split factor. Starting from Sivatski's example, we construct examples of algebras with orthogonal involution, of index $2^r$ and co-index $2^k$, where $r\geq 2$ and $k\geq 1$ are arbitrary, that are not totally decomposable with $k$ split factors as algebras with involution, even though the underlying algebra does admit a total decomposition with $k$ split factors (see Remark \ref{rem5.1}). Those algebras have center some iterated Laurent power series fields, and the theory of gauges play a crucial role in the proofs.

Assume now that the involution is of symplectic type. We give a partial answer to our main question in this case. Since we assume $\sigma$ is anisotropic, the index of $A$ is at least $2$. In index $2$, Becher showed in ~\cite{Bec08} that $(A,\sigma)\simeq  (H',\gamma)\otimes (M_{2^{m-1}}(F), \ad_{\pi''})$ where $H'$ is some quaternion algebra, $\gamma$ is the canonical involution on $H'$ and $\pi''$ is some $(m-1)$-fold Pfister form over $F$. In degree $8$ and index $4$, we prove that $(A,\sigma)$ is totally decomposable with $1$ split factor (see Prop.~\ref{degree8}). Finally, in index $2^r$ and co-index $2^k$ for arbitrary $r\geq 3$ and $k\geq 1$, we give examples of $(A,\sigma)$ such that $A$ is totally decomposable with $k$ split factors but $(A, \sigma)$ is not totally decomposable with $k$ split factors (see Remark \ref{rem5.1}). 

The paper is organized as follows:  in Section \ref{deg8} we deal with the degree $8$ algebras with involution. In particular, we show that if $(A,\sigma)$ is a central simple algebra of degree $8$ and co-index  $2^k$ (where $k\le 2$) with a symplectic involution then $(A,\sigma)$ is totally decomposable with $k$ split factors (see Prop. \ref{degree8}). Section \ref{gauge} collects preliminary results on gauges, which are the main tool in the remaining part of the paper.  We give in Sections \ref{fixIndex} and \ref{largeIndex} the proof of our main results and the examples that extend Sivatski's result to algebras with involution of arbitrarily large co-index (see Cor. \ref{cor4.1} and Remark \ref{rem5.1}) and to algebras with involution of arbitrarily large index (see Cor. \ref{cor5.1} and Remark \ref{rem5.1}).  Corollaries \ref{cor4.1} and \ref{cor5.1} are the immediate consequences of the main Theorems \ref{thm2.0} and \ref{thm2.1} below.
\subsection{Statement of main results}\label{statement}
Throughout this paper the characteristic of the base field $F$ is assumed to be different from $2$ and all algebras are associative and finite-dimensional over their centers.
We will use freely the standard terminology and notation from the theory of finite-dimensional algebras, the theory of valuations on division algebras and the theory of involutions on central simple algebras. For these, as well as background information, we refer the reader to ~\cite{Pie82} and ~\cite{KMRT98}.

Let us fix some notations: let $F$ be a field, consider the fields  
\[
K= F((t))  \quad \text{and} \quad L=F((t_1))((t_2)),
\]
where $t, t_1, t_2$ are independent indeterminates over $F$, and the quaternion algebra  $Q=(t_1, t_2)$ over $L$. For an $n$-dimensional quadratic form $q$, following ~\cite{Bec08}, we denote by $\Ad_q$ the split algebra with orthogonal involution $(M_n(F), \ad_q)$ where $\ad_q$ is the involution on $M_n(F)$ that is adjoint to the polar form of $q$. Let $S$ be a central simple algebra over $F$  and let $\rho$ be an anisotropic involution on $S$. Let $\rho'$ be any  involution of orthogonal or symplectic type on $Q$ and let
\[
(S_1, \sigma)= (S,\rho)\otimes_F \Ad_{\langle\langle t\rangle\rangle}, \quad  (S_2,\tau)= (S,\rho)\otimes_F (Q,\rho').
\]
\begin{theorem}\label{thm2.0}
If $(S_1,\sigma)$ is totally decomposable with $(n+1)$ split factors then $(S,\rho)$ is totally decomposable  with $n$ split factors.
\end{theorem}
\begin{theorem}\label{thm2.1}
If $(S_2,\tau)$ is totally decomposable with $n$ split factors then $(S,\rho)$ is totally decomposable  with $n$ split factors.
\end{theorem}
Notice that if $(S,\rho)$ is totally decomposable then the algebras with involution $(S_1,\sigma)$ and $(S_2,\tau)$ are totally decomposable. But we have:
\begin{remark}
Decompositions of the algebra $(S_1,\sigma)$ of Thm.~\ref{thm2.0} do not necessarily arise from a decomposition of $(S,\rho)$ by multiplication with the factor $\Ad_{\langle\langle t\rangle\rangle}$. For instance, the arguments in the proof of Lemma~\ref{lm3.2} show that if $(S,\rho)$ is a quaternion algebra $(a,b)$ with its canonical involution, then
\[
(S,\rho)\otimes\Ad_{\langle\langle t\rangle\rangle} \simeq ((at,b),\nu)\otimes((t,b),\gamma)
\]
where $\gamma$ is the canonical involution and $\nu$ an orthogonal involution with discriminant $at$. Likewise, decompositions of the algebra $(S_2,\tau)$ of Thm.~\ref{thm2.1} do not necessarily arise from a decomposition of $(S,\rho)$.
\end{remark}
\section{Degree 8}\label{deg8}
Let $A$ be a central simple algebra  of  degree $8$ and co-index $2$ and let $\sigma$ be an anisotropic involution on $A$. If $\sigma$ is orthogonal, it follows from ~\cite[Thm. 1.1]{MQT09} that $(A,\sigma)\simeq (D, \theta)\otimes (M_2(F), \text{ad}_\varphi)$ where $(D,\theta)$ is some biquaternion algebra with orthogonal involution and ad$_\varphi$ is the adjoint involution with respect to some binary quadratic form $\varphi$. Notice that $(D,\theta)$ need not to be decomposable in general. Indeed, the following result is due to Sivatski.
\begin{proposition}[{\cite[Prop. 5]{Siv05}}]\label{sivatski}
There exists a totally decomposable $F$-algebra with orthogonal involution $(A',\sigma')$ of degree $8$ and co-index $2$ such that $(A',\sigma')$ is not totally decomposable with $1$ split factor.
\end{proposition}
Now assume that $\sigma$ is symplectic. It is shown in ~\cite{BMT03} that  $(A,\sigma)\simeq (Q, \gamma)\otimes_F (A_0, \sigma_0)$ where $(Q, \gamma)$ is a quaternion subalgebra of $A$ with  the canonical involution and $\sigma_0$ is an orthogonal involution on the subalgebra  $A_0$ of $A$. Contrary to the orthogonal case, we show in the following proposition that there  exists no analogue of the above result of Sivatski when $\sigma$ is symplectic. First, let us recall that a relative invariant for symplectic involutions on central simple algebras of degree divisible by $4$ is defined in \cite{BMT03} and the authors show that the relative invariant is trivial for an algebra of degree $8$ and co-index $2$ with symplectic involution $(A,\sigma)$ if and only if $(A,\sigma)$ is totally decomposable with $1$ split factor (\cite[Thm. 8]{BMT03}).  In ~\cite{GPT09}, this invariant leads to an absolute  invariant $\Delta$, with value in $H^3(F, \mu_2)$, for symplectic involutions. For a central simple algebra $A$ of degree $8$ with a symplectic involution $\sigma$, the element $\Delta(A,\sigma)$ is zero if and only if $(A,\sigma)$ is totally decomposable (see ~\cite[Thm. B]{GPT09}).
\begin{proposition}\label{degree8}
Let $(A,\sigma)$ be a totally decomposable algebra of degree $8$ with an anisotropic symplectic involution, of co-index $2^m$ (where $m\le 2$). Then $(A,\sigma)$ is totally decomposable with $m$ split factors.
\end{proposition}
\begin{proof}
If the co-index of $A$ is $4$, it follows from Becher's result ~\cite[Cor]{Bec08} that $(A,\sigma)$ is totally decomposable with $2$ split factors.  Assume that the co-index of $A$ is $2$. Since $(A,\sigma)$ is totally decomposable, the absolute invariant $\Delta(A,\sigma)$, defined in ~\cite{GPT09}, is zero. Therefore it readily follows by  ~\cite[Thm. 8]{BMT03} that $(A,\sigma)$ is totally decomposable with $1$ split factors.
\end{proof}
\section{Gauges on algebras with involution}\label{gauge}
Gauges on algebras with  involution play a major role in the next sections. In this section we recall the notions of value functions and gauges introduced in \cite{RTW07}, \cite{TW10},   \cite{TW11} and  \cite{TW14} and gather some results for the sequel.

We fix a divisible totally ordered abelian group $\Gamma$, which will contain the value of all the valuations and the degree of all the gradings we consider. Thus, a valued division algebra is a pair $(D,v)$ where $D$ is a division algebra and $v\,:\, D \longrightarrow \Gamma\cup\{\infty\}$ is a valuation. The group $v(D\m)$ of values of $D$ is denoted by $\Gamma_D$, and the residue division algebra by $\overline D$. The valuation $v$ defines a filtration on $D$: for $\gamma\in \Gamma$, set
\[
D_{\geq \gamma}=\{d\in D\,|\, v(d)\geq \gamma\},\quad D_{> \gamma}=\{d\in D\,|\, v(d)> \gamma\}  
\]
and
\[
D_\gamma= D_{\geq \gamma}/D_{> \gamma}.
\]
The associated graded ring (of $v$ on $D$) is
\[
\gr_v(D)=\bigoplus_{\gamma\in \Gamma} D_\gamma
\]
with the multiplication induced by the multiplication in $D$. The homogeneous elements of $\gr_v(D)$ are those in $\bigcup_{\gamma\in \Gamma} D_\gamma$. For $d\in D\m$, we write $\widetilde d$ for the image $d+D_{>v(d)}$ for $d\in D_{v(d)}$, and $\widetilde{0}_D=0$ in $\gr_v(D)$. It follows from the fact $v(cd)=v(c)+v(d)$, for $c,d\in D\m$, that $\widetilde{cd}=\widetilde c \widetilde d$. So, in particular, we have $\widetilde d \,\widetilde{(d\inv)}=\widetilde 1$ for any $d\in D\m$. This shows that $\gr_v(D)$ is graded division ring. The grade group of $\gr_v(D)$, denoted $\Gamma_{\gr_v(D)}$, is $\{\gamma\in \Gamma\, |\, D_\gamma\ne 0\}$. Notice that $\Gamma_{\gr_v(D)}=\Gamma_D$ and the degree $0$ component of $\gr_v(D)$ is $D_0=D_{\geq 0}/D_{>0}= \overline D$.

Let $(F,v)$ be a valued field, and let $M$ be a (right) $F$-vector space. A $v$-\emph{value function} on $M$ is a map $\alpha\,:\, M\longrightarrow \Gamma\cup\{\infty\}$ such that 
\begin{itemize}
\item[(i)] $\alpha(x)=\infty$ if and only if $x=0$;
\item[(ii)] $\alpha(x+y)\geq \min(\alpha(x),\alpha(y))$ for $x,y\in M$;
\item[(iii)] $\alpha(xc)=\alpha(x)+v(c)$ for all $x\in M$ and $c\in F$.
\end{itemize}
If $\alpha$ is a $v$-value function on $M$, in the same way as for the construction of the graded division algebra, we associate to $M$ a graded $\gr_v(F)$-module that we denote by $\gr_\alpha(M)$. In addition, if $M$ is finite-dimensional and $[M:F]=[\gr_\alpha(M):\gr_v(F)]$, we say that $\alpha$ is a \emph{norm} (or a $v$-norm). A $v$-value function $g$ on an $F$-algebra $A$ is \emph{surmultiplicative} if
\[
g(1)=0 \quad \text{and}\quad g(xy)\geq g(x)+g(y) \,\, \text{for all } x,y\in A.
\]
If $g$ is a surmultiplicative $v$-value function on an $F$-algebra $A$, then $\gr_g(A)$ is an algebra over $\gr_v(F)$, in which multiplication is defined by
\[
\widetilde a\widetilde b= ab+A_{>g(a)+g(b)}=
\left\{ \begin{array}{ll}
\widetilde{ab} & \text{if $g(ab)=g(a)+g(b)$},\\
0 & \text{if $g(ab)> g(a)+g(b)$},
  \end{array} \right.
\quad \text{for $a,b\in A$}.
\] 
Suppose $A$ is a finite-dimensional simple $F$-algebra. A surmultiplicative $v$-value function $g$ on $A$ is called an $F$-\emph{gauge} (or a $v$-gauge) if:
\begin{itemize}
\item[(i)] $g$ is a $v$-norm, i.e, $[A:F]=[\gr_g(A):\gr_v(F)]$;
\item[(ii)] $\gr_g(A)$ is a graded semisimple $\gr_v(F)$-algebra, i.e, $\gr_g(A)$ does not contain any non-zero nilpotent homogeneous two-sided ideal.
\end{itemize}
We give the following examples that will be used in the proof of our main results.
\begin{examples}\label{example}
The algebras $S_1$ and $S_2$ are as in Theorems \ref{thm2.0} and \ref{thm2.1}.

(1) Let
\[
i= \begin{pmatrix}
1 & 0\\
0 &-1
\end{pmatrix} , \quad 
j=\begin{pmatrix}
0 & t\\
1 & 0
\end{pmatrix}\in M_2(K),
\]
so $\ad_{\langle\langle t\rangle\rangle}(i)=i$ and $\ad_{\langle\langle t\rangle\rangle}(j)=-j$. Let $(s_\ell)_{1\le \ell \le n}$ be an $F$-basis of $S$, so $(s_\ell \otimes 1,  s_\ell \otimes i,  s_\ell \otimes j,  s_\ell \otimes ij )_{1\le \ell \le n}$ is a $K$-basis of $S_1$. The $t$-adic valuation $v$ on $K$ extends to a map
\[
g_1\,:\, S_1\longrightarrow\, (\tfrac{1}{2} \Z)\cup\{\infty\}
\]
defined by
\[
g_1\big(\sum_{1\le \ell \le n} s_\ell \otimes(\alpha_\ell+ \beta_\ell i+ \gamma_\ell\ j+ \delta_\ell ij) \big)=\min_{1\le \ell \le n}\big(v(\alpha_\ell), v(\beta_\ell), v(\gamma_\ell)+\tfrac{1}{2}, v(\delta_\ell)+\tfrac{1}{2}\big).
\]
We readily check that $g_1(a_1a_2)\geq g_1(a_1)+g_1(a_2)$ for $a_1, a_2\in S_1$.  It then follows by ~\cite[Lemma 3.23]{TW14} that $g_1$ defines a filtration on $S_1$.  The associated graded ring is noted $\gr_{g_1}(S_1)$. Let $\widetilde t$ be the image of $t$ in $\gr_v(K)$. We have
\[
\gr_{g_1}(S_1) =  S\otimes_F M_2(F[\widetilde t, {\widetilde t\,}^{-1}])
\]
with the grading defined by
\[
\begin{array}{ll}
\gr_{g_1}(S_1)_\gamma =
\begin{pmatrix}
S\widetilde{t\,\,}^\gamma & 0\\
0 & S \widetilde{t\,\,}^\gamma
\end{pmatrix} &\qquad  \text{ for } \gamma\in\Z\\
\gr_{g_1}(S_1)_\gamma = 
\begin{pmatrix}
0 & S\widetilde{t\,\,}^{\gamma+\tfrac{1}{2}}\\
S\widetilde{ t\,\,}^{\gamma-\tfrac{1}{2}} & 0
\end{pmatrix} & \qquad \text{ for } \gamma\in\tfrac{1}{2}\Z\backslash\Z.
  \end{array} 
\]
Let $x$ be a nonzero homogeneous element in $\gr_{g_1}(S_1)$. Since  $\widetilde t$ is invertible in $\gr_{g_1}(S_1)$, the two-sided ideal $I(x)\subset \gr_{g_1}(S_1)$ generated by $x$ contains a nonzero element of $M_2(S)$. But $M_2(S)$ is simple and contained in $\gr_{g_1}(S_1)$, hence $I(x)$ contains $1$. Therefore, $\gr_{g_1}(S_1)$ is graded  simple algebra, and  $g_1$ is a gauge on $S_1$ (see  ~\cite[Def. 3.31]{TW14}).

(2) The $(t_1, t_2)$-adic valuation $v$  on $L$ being Henselian, it extends to a valuation $w$ on $Q$ defined by
\[
w(q)= \tfrac{1}{2}v\big(\Nrd(q)\big)\in \big(\tfrac{1}{2}\Z\big)^2\cup \{\infty\} \qquad \text{for } q\in Q,
\]
where $\Nrd$ is the reduced norm. One further extends $w$ to a map
\[
g_2\,:\, S_2\, \longrightarrow\,  \big(\tfrac{1}{2}\Z\big)^2\cup \{\infty\}
\] 
as follows: we consider the $F$-base $(s_\ell)_{1\le \ell \le n}$ of $S$ as above. Every element in $S_2$ has a unique representation of the form $\sum_\ell s_\ell\otimes q_\ell$ for some $q_\ell\in Q=(t_1, t_2)$. We define
\[
g_2(\sum_\ell s_\ell\otimes q_\ell)= \min\{ w(q_\ell)\,\, |\,\, 1\le \ell\le n  \}.
\]
It is easy to see that this definition does not depend on the choice of the base $(s_\ell)_{1\le \ell\le n}$. We easily check that $g_2$ is a gauge. The associated graded ring is 
\[
\gr_{g_2}(S_2)=S\otimes_F (\widetilde t_1, \widetilde t_2)_{\gr_v(L)}
\]
where $\widetilde t_1$ and $\widetilde t_2$ are respectively the images of $t_1$ and $t_2$ in  $\gr_v(L)=F[\widetilde t_1, \widetilde {t_1\,}^{-1}, \widetilde t_2, \widetilde {t_2\,}^{-1}]$.
\end{examples}

Now, let $\theta$ be an $F$-linear involution on a central simple algebra $A$ over $F$. Assume that $v$ is a valuation on $F$ and $A$ carries a $v$-gauge $g$. The gauge $g$ is said to be \emph{invariant} under $\theta$ (or $\theta$-invariant) if
\[
g(\theta(x))=g(x) \quad \text{for all $x\in A$}.
\]
In this case, the involution $\theta$ preserves the filtration on $A$ defined by $g$. Thus $\theta$ induces an involution $\widetilde \theta$ on $\gr_g(A)$ such that $\widetilde\theta(\widetilde x)=\widetilde{\theta(x)}$ for all $x\in A$. The involution $\theta$ is said to be anisotropic if there is no nonzero elements $x\in A$ such that $\theta(x)x=0$ (see ~\cite[\S 6.A]{KMRT98}). Likewise, $\widetilde\theta$ is said to be anisotropic if there is no nonzero homogeneous element $\varepsilon\in \gr_g(A)$ such that $\widetilde\theta(\varepsilon)\varepsilon=0$. Clearly, if $\widetilde\theta$ is anisotropic then $\theta$ is anisotropic. If $g$ is a $v$-gauge on $A$ that is invariant under $\theta$, we denote by $\theta_0$ the $0$ component of $\widetilde\theta$. Thus, $\theta_0$ is an involution on the $\overline F$-algebra $A_0=A_{\geq 0}/A_{>0}$, which may be viewed as the residue involution on the residue algebra $A_0$. 

Let $(F,v)$ be a Henselian valued field with $\operatorname{char}(\overline F)\ne 2$ and let $(A,\theta)$ be a central simple $F$-algebra with ($F$-linear) involution.  If the involution $\theta$ is anisotropic, it is shown in ~\cite[Thm. 2.2]{TW11} that there exists a unique $v$-gauge $g$ on $A$ such that $g$ is invariant under $\theta$ and $g(\theta(a)a)=2g(a)$ for all $a\in A$. Following Tignol-Wadsworth's terminology (\cite{TW11}) such a gauge $g$ is called $\theta$-\emph{special}. 

The following proposition gives an explicit description of the special gauges on $(S_1,\sigma)$ and $(S_2,\tau)$. 
\begin{proposition}\label{prop01}
We consider the gauges $g_1$ and $g_2$ constructed in Examples \ref{example} and the algebras with involution $(S_1, \sigma)$ and $(S_2,\tau)$ are  as in Theorems \ref{thm2.0} and \ref{thm2.1}.
\begin{itemize}
\item[(1)] The gauge $g_1$ is the unique gauge on $S_1$ that is invariant under $\sigma$, and moreover it satisfies $g_1(\sigma(a)a)=2g_1(a)$ for all $a\in S_1$.
\item[(2)] The gauge $g_2$ is the unique gauge on $S_2$ that is invariant under $\tau$, and moreover it satisfies $g_2(\tau(b)b)=2g_2(b)$ for all $b\in S_2$.
\end{itemize}
(That is, $g_1$ and $g_2$ are $\sigma$-special and $\tau$-special respectively.)
\end{proposition}
\begin{proof}
(1)  It follows from the definition of  $g_1$ that it is invariant under $\sigma$, that is, $g_1(\sigma(a))=g_1(a)$ for all $a\in S_1$. On $\gr_{g_1}(S_1)_0=S\times S$, the induced involution is  $\rho\times \rho$. Therefore, since $\rho$ is anisotropic, $\sigma$ also is by ~\cite[Cor. 2.3]{TW11}, and the statement follows by   ~\cite[Thm. 2.2 and Cor. 2.3]{TW11}.

(2) One may  check  that $g_2(\tau(b))=g_2(b)$ for all $b\in S_2$ (that is, $g_2$ is invariant under $\tau$), and $g_2(\tau(b)b)=2g_2(b)$. We may therefore consider the residue algebra with involution $\gr_{g_2}(S_2, \tau)_0$. Since the residue of $Q$ is $F$ and $S$ is defined over $F$, we have $\gr_{g_2}(S_2, \tau)_0\simeq (S,\rho)$.  The involution $\rho$ being anisotropic, the assertion follows by the same arguments as above.
\end{proof}
\section{Decomposition with arbitrarily large co-index}\label{fixIndex}
The goal of this section is to give examples of a totally decomposable algebras with orthogonal involution of degree $2^m$ and co-index $2^{m-2}$ (for $m\geq 3$) which are not totally decomposable with $(m-2)$ split factors. We start out by proving Thm.~\ref{thm2.0}. We first prove the following:
\begin{lemma}\label{lm3.1}
Let $(H,\nu)$ be a quaternion algebra with anisotropic involution over $K=F((t))$ and let $\gamma$ denote the canonical involution on $H$. 
\begin{itemize}
\item[(1)] Suppose the algebra $H$ is not defined over $F$. Then $H\simeq (at,b)$ for some $a,b\in F^\times$.

  \item[(2)] Suppose the algebra with involution $(H,\nu)$ is not defined over $F$ and  $\nu$ is orthogonal and let $\disc(\nu)$ denote the discriminant of $\nu$. Then either $\disc(\nu)=at\cdot K^{\times 2}$ for some $a\in F^\times$ or $\disc(\nu)=b\cdot K^{\times 2}$ for some $b\in F^\times$.
\begin{itemize}
\item[--] If  $\disc(\nu)=at\cdot K^{\times 2}$ (resp. $b\cdot K^{\times 2}$) then there exist quaternion generators $i,j$ of $H$ such that $i^2=at$,  $j^2=b$ and $\nu=\Int(i)\circ \gamma$ (resp. $\nu=\Int(j)\circ \gamma$).
\item[--] If $H$ is split, then $(H, \nu)=\Ad_{\langle\langle at \rangle\rangle}$.
\end{itemize}
\end{itemize}
\end{lemma}
\begin{proof}
We first recall that (see ~\cite[Cor. VI.1.3]{Lam05}) 
\[
K\m/K^{\times 2}=\{\alpha \cdot K^{\times 2}\, \vert\, \alpha \in F\}\cup\{\alpha t\cdot K^{\times 2}\, \vert\, \alpha\in F\}.
\]
(1) Let $q$ be a pure quaternion in $H$;  we have $q^2\in K\m/K^{\times 2}$. Up to scaling, we may assume that either $q^2= a t$ for some $a \in F\m$ or $ q^2=b$ for some $b\in F\m$. Since we assume that $H$ is not defined over $F$, the algebra $H$ is either isomorphic to $H\simeq (at, b)$ or to $H\simeq (at, bt)\simeq (at, -ab)$.

(2) Now assume that $\nu$ is orthogonal. By ~\cite[(2.7)]{KMRT98}, the involution $\nu$ has the form $\Int(s)\circ \gamma$ for some invertible $s\in H$ with $\gamma(s)=-s$.  \cite[(7.3)]{KMRT98} shows that  $\disc(\nu)=-\Nrd(s)\cdot K^{\times 2}\in K\m/K^{\times 2}$, where $\Nrd$ denotes the reduced norm. Therefore, either $\disc(\nu)= at \cdot K^{\times 2}$ for some $a\in F^\times$ or $\disc(\nu)=b \cdot K^{\times 2}$ for some $b\in F^\times$. If $\disc(\nu)=at \cdot K^{\times 2}$, we choose $i=s$.  One has $\nu=\Int(i)\circ \gamma$. Let $j\in H^\times$ be such that $j^2=y\in K\m/K^{\times 2}$ and $ij=-ji$ (recall that such a $j$ always exists). Up to scaling we may assume $y\in F\m$ or $y\in F\m\cdot t$.  We have the isomorphism $(H,\nu)\simeq ((at, b),\Int(i)\circ \gamma)$ if $y=b\in F\m$ and $(H,\nu)\simeq ((at, bt),\Int(i)\circ \gamma)$ if $y$ has the form $y=bt$. In the latter case, we may substitute $ij^{-1}$ for $j$ and reduce to the case where $j^2\in F\m$.

Likewise, if $\disc(\nu)=b \cdot K^{\times 2}$ for some $b\in F\m$, we choose $j=s$. So  $\nu= \Int(j)\circ \gamma$. Let $i\in H\m$ be such that $i^2\in K\m/K^{\times 2}$ and $ij=-ji$. Since $(H,\nu)$ is not defined over $F$, up to scaling we have necessarily $i^2=at$ for some $a\in F\m$. This yields that $(H,\nu)\simeq ((at, b),\Int(j)\circ \gamma)$.

Now, suppose $H$ is split. If so, $(H,\nu)\simeq \Ad_{\langle 1, -\disc(\nu) \rangle}$. Since it is not defined over $F$, we get $\disc(\nu)= at\cdot K^{\times 2}$ for some $a\in F\m$, that is, $(H,\nu)\simeq \Ad_{\langle\langle at \rangle\rangle}$.
 \end{proof}
\begin{lemma}\label{lm3.2}
Let $(H_1,\nu_1)$, $(H_2,\nu_2)$ be quaternion algebras with anisotropic involutions over $K=F((t))$.
There exist quaternion algebras with involution $(H'_1, \nu'_1)$, $(H'_2, \nu'_2)$ with $(H'_1, \nu'_1)$ defined over $F$ such that
\[
(H_1, \nu_1)\otimes (H_2, \nu_2) \simeq (H'_1, \nu'_1)\otimes (H'_2, \nu'_2).
\]
Moreover,
\begin{itemize}
\item[--] if $(H_1,\nu_1)$ and $(H_2,\nu_2)$ are split, we may find $(H'_1,\nu'_1)$ and $(H'_2,\nu'_2)$ such that both are split;
\item[--] if $(H_1,\nu_1)$ is split and $(H_2,\nu_2)$ is not split, we may find $(H'_1,\nu'_1)$ and $(H'_2,\nu'_2)$ such that $(H'_1,\nu'_1)$ is split and $(H'_2,\nu'_2)$ is not split.
\end{itemize}
\end{lemma}
\begin{proof}
If $(H_1,\nu_1)$ or $(H_2,\nu_2)$ is defined over $F$, there is nothing to show. Otherwise, write $H_1=(a_1t, b_1)$ and $H_2=(a_2t, b_2)$ for some  $a_1, a_2, b_1, b_2\in F$. By Lemma \ref{lm3.1}, there exist quaternion generators  $i_1, j_1$ and $i_2, j_2$  of $H_1$ and $H_2$ respectively such that  $i_\ell^2=a_\ell t, j_\ell^2=b_\ell$,  and $\nu_\ell(i_\ell)=\pm i_\ell$ and $\nu_\ell(j_\ell)=\pm j_\ell$ for $\ell= 1, 2$. Let $H_1', H_2'\subset H_1\otimes H_2$ be quaternion subalgebras generated by $t\inv i_1\otimes i_2$, $j_1\otimes 1$ and $1\otimes i_2$, $j_1\otimes j_2$ respectively. In fact, $H_1'=(a_1a_2, b_1)$ and $H_2'=(a_2t, b_1b_2)$. We easily check that $H_1'$ and $H_2'$ are stable under $\nu_1\otimes \nu_2$. Denote by $\nu'_\ell$ (for $\ell=1,2$) the restriction of $\nu_1\otimes \nu_2$ to $H'_\ell$.  We have the isomorphism
\[
(H_1,\nu_1)\otimes (H_2,\nu_2)\simeq (H'_1,\nu'_1)\otimes (H'_2,\nu'_2) 
\]
with $(H'_1,\nu'_1)$ defined over $F$.

Now, if $(H_1,\nu_1)$ and $(H_2,\nu_2)$ are split, then the quadratic forms $\langle 1, -b_\ell \rangle\perp \langle -a_\ell t \rangle\langle 1, -b_\ell \rangle$ are isotropic (for $\ell=1,2$), hence $b_1$ and  $b_2$ are squares.  Thus, the algebras $H_1'=(a_1a_2, b_1)$ and $H_2'=(a_2t, b_1b_2)$ are split. Finally, if $(H_1,\nu_1)$ is split and $(H_2,\nu_2)$ is not split, then $b_1\in F^{\times 2}$ and $b_2\notin F^{\times 2}$. So $H'_1$ is split and $H'_2$ is not split.  This concludes the proof.
\end{proof}
\begin{proposition}\label{prop3.1}
The algebra with involution $(S_1,\sigma)$ is as Thm.~\ref{thm2.0}. Assume $(S_1,\sigma)$ is totally decomposable with $(n+1)$ split factors. Then $(S_1,\sigma)\simeq \Ad_{\langle\langle at\rangle\rangle}\otimes (S',\sigma')$ for some $a\in F$ and some totally decomposable algebra with involution $(S',\sigma')$ with $n$ split factors. Moreover, each factor of $(S',\sigma')$ is defined over $F$.
\end{proposition}
\begin{proof}
Write $(S_1,\sigma)\simeq \bigotimes_{k=1}^{m}(Q_k,\sigma_k)$. It follows by part (1) of Lemma \ref{lm3.1} that each $Q_k$ is either defined over $F$ or isomorphic to a quaternion algebra of the form $(at, b)$ for some $a,b\in F$. Applying repeatedly   Lemma \ref{lm3.2}, we get a total decomposition of $(S_1,\sigma)$ in which each factor is defined over $F$ except only one factor.  

We show that the number of  split factors in the decomposition of  $(S_1,\sigma)$ does not change by applying  Lemma \ref{lm3.2}. Without loss of generality we may assume that $(S_1,\sigma)$ is not totally decomposable with more than $(n+1)$ split factors. Let $(Q_i, \sigma_i)$ and $(Q_j, \sigma_j)$ be two factors of $(S_1,\sigma)$. If $(Q_i, \sigma_i)$ and $(Q_j, \sigma_j)$ both are split,  Lemma \ref{lm3.2} shows that the tensor product $(Q_i, \sigma_i)\otimes(Q_j, \sigma_j)$ is exchanged for another one in which each factor is split. If $(Q_i, \sigma_i)$ is split and $(Q_j, \sigma_j)$ is not split, again Lemma \ref{lm3.2} shows that the product $(Q_i, \sigma_i)\otimes(Q_j, \sigma_j)$ is exchanged for a product in which only one factor is split. Finally, suppose $(Q_i, \sigma_i)$ and $(Q_j, \sigma_j)$ both are division algebras with involution. Since we assume that $(S_1,\sigma)$ is not totally decomposable with more than $(n+1)$ split factors, the product $(Q_i, \sigma_i)\otimes(Q_j, \sigma_j)$ is necessarily exchanged for another one in which each factor is a division algebra with involution. It then follows that the number of  split factors in the decomposition of  $(S_1,\sigma)$ does not change by applying  Lemma \ref{lm3.2}. 

Let us denote by $((at,b), \sigma_r)$ the   factor of $(S_1,\sigma)$ which is not defined over $F$.  We claim that  $((at,b), \sigma_r)\simeq \Ad_{\langle\langle a't\rangle\rangle}$ for some $a'\in F\m$. To see this, we first show that $(at, b)$ is split.  Assume the contrary, that is, $(at, b)$ is a division quaternion algebra over $K$. On the one hand, according to Prop.~\ref{prop01}, there exists a unique gauge $g_1$ on $S_1$ that is invariant under $\sigma$ and its residue algebra is $\gr_{g_1}(S_1)_0=S\times S$ with the induced involution $\rho\times \rho$. On the other hand, write $(S_1,\sigma)\simeq ((at, b),\sigma_r)\otimes (S',\sigma')$ where $S'$ is the centralizer of $(at, b)$ in $S_1$. As we showed above, $(S',\sigma')$ is totally decomposable with $n$ split factors and each factor of $(S',\sigma')$ is defined over $F$. Now, we construct the same gauge on $S_1$ by taking into account the fact that $(at, b)$ is a division algebra: the field $K$ being Henselian, the $t$-adic valuation  $v$ of $K$ extends to a valuation $w$ on $(at, b)$. Let $(a_i)_{1\le i\le s}$ be an $F$-basis of $S'$. So every element in $S_1$ has a unique representation $\sum_i q_i \otimes a_i$ for some $q_i\in (at, b)$. It is easy to verify that the map $g_3:\, S_1\longrightarrow \tfrac{1}{2}\Z\cup\{\infty\}$ defined by 
\[
g_3(\sum_i q_i \otimes a_i)= \min\{w(q_i)\,\vert\, 1\le i\le s\}
\]
is a $v$-gauge on $S_1$. Moreover, it follows from the definition  that $g_3$ is invariant under $\sigma$ and $g_3(\sigma(a)a)=2g_3(a)$ for all $a\in S_1$. Hence, $g_3$ is the unique $v$-gauge on $S_1$ that is preserved by $\sigma$ by the same arguments as above. Therefore $g_3=g_1$ by uniqueness arguments. This leads to a contradiction, since the residue  of $S_1$ with respect to $g_3$ is  $\gr_{g_3}(S_1)_0\simeq A'\otimes F(\sqrt{b})$, which is simple, as opposed to the residue of $S_1$ with respect to $g_1$, which is not simple.  Thus the quaternion algebra $(at,b)$ is split. Since the involution $\sigma$ is anisotropic, $\sigma_r$ is necessarily orthogonal. It Then follows by part (2) of Lemma \ref{lm3.1} that $((at,b), \sigma_r)\simeq \Ad_{\langle\langle a't\rangle\rangle}$ for some $a'\in F\m$. Therefore $(S_1,\sigma)\simeq \Ad_{\langle\langle a't\rangle\rangle}\otimes (S',\sigma')$. We have the stated description of  $(S_1,\sigma)$.
 \end{proof}
\begin{proof}[Proof of Theorem \ref{thm2.0}]
Assume that $(S_1,\sigma)$ is totally decomposable with $(n+1)$ split factors. By Prop.~\ref{prop3.1}, $(S_1,\sigma)\simeq \Ad_{\langle\langle at\rangle\rangle}\otimes (S',\sigma')$, where $(S',\sigma')$ is totally decomposable with $n$ split factors, and each factor of $(S',\sigma')$ is defined over $F$. We write  $(S',\sigma')\simeq \otimes_{\ell=2}^m(Q_\ell, \sigma_\ell)\otimes K$. Consider the gauge $g_1$ on $S_1$ from  Prop.~\ref{prop01}, and recall that the residue algebra with involution  of $(S_1,\sigma)$ is $\gr_{g_1}(S_1,\sigma)_0\simeq (S\times S, \rho\times \rho)$. For $\ell=2,\ldots, m$, let $i_\ell, j_\ell\in S'$ be such that 
\begin{equation}\label{eq2.2}
i_\ell^2=a_\ell,\quad j_\ell^2=b_\ell, \quad i_\ell j_\ell=-j_\ell i_\ell,\quad \sigma'(i_\ell)=\pm i_\ell,\quad \sigma'(j_\ell)=\pm j_\ell,
\end{equation}
so $Q_\ell=(a_\ell, b_\ell)$ for all $\ell$. Since $g_1(\sigma(a)a)=2g_1(a)$ for all $a\in S_1$,  we have $g_1(i_\ell)= \tfrac{1}{2}g_1(\pm a_\ell)=0$ for all $\ell$.  Similarly, $g_1(j_\ell)= \tfrac{1}{2}g_1(\pm b_\ell)=0$ for all $\ell$. The images  $\widetilde {i_\ell}, \widetilde{j_\ell}$ of $i_\ell, j_\ell$ in $\gr_{g_1}(S_1)_0\simeq S\times S$ satisfy conditions similar to $(\ref{eq2.2})$.  The involution $\sigma$ being preserved by $g_1$, it induces an involution $\widetilde \sigma$ on $\gr_{g_1}(S_1)$. Consider a projection $\gr_{g_1}(S_1)_0\to S$. This projection is a homomorphism of algebras with involution $p: \, (\gr_{g_1}(S_1), \widetilde\sigma)_0\to (S,\rho)$. The images $p(\widetilde{i_\ell}), p(\widetilde{j_\ell})$ generate a copy of $\otimes_{\ell=2}^m(Q_\ell,\sigma_\ell)$ in $(S,\rho)$. Thus, we have the isomorphism $\otimes_{\ell=2}^m(Q_\ell,\sigma_\ell)\simeq (S,\rho)$ by dimension count. Therefore $(S, \rho)$ is totally decomposable with $n$ split factors.
\end{proof}
\begin{corollary}\label{cor4.1}
Let $(A',\sigma')$ be the algebra with orthogonal involution of Prop.~\ref{sivatski}. The algebra with involution $\Ad_{\langle\langle t\rangle\rangle}\otimes (A',\sigma')$ is not totally decomposable with $2$ split factors.
\end{corollary}
\begin{proof}
Assume the contrary, that is, $\Ad_{\langle\langle t\rangle\rangle}\otimes (A',\sigma')$ is totally decomposable with $2$ split factors. It follows by Thm.~\ref{thm2.0} that $(A',\sigma')$ is totally decomposable with $1$ split factor; that is impossible. Hence the corollary is proved.
\end{proof}
\begin{remark}\label{rem4.1}
The construction in Cor.~\ref{cor4.1} can be iterated to obtain examples of  algebras with orthogonal involution of degree  $2^m$ (for $m \geq 4$) and co-index $2^{m-2}$  that  are not totally decomposable with $(m-2)$ split factors although the underlying algebras are totally decomposable with $(m-2)$ split factors.
\end{remark}
\section{Decomposition with arbitrarily large index}\label{largeIndex}
The main object of this section is to prove Thm.~\ref{thm2.1}. This theorem leads to  examples of  totally decomposable algebras with involution of degree $2^m$ (for $m\geq 3$) and co-index $2$  which are not totally decomposable with $1$ split factor.  For the proof of Thm.~\ref{thm2.1}  we need some preliminary  results. Let $E$ be an arbitrary central simple algebra of exponent $2$ over $F$. As in part~\ref{statement} of Section \ref{section1}, $Q$ denotes the quaternion algebra $(t_1,t_2)$  over the field $L=F((t_1))((t_2))$. We set 
\[
C=E\otimes_F Q.
\]
Consider the gauge as constructed in  part (2) of Example \ref{example}, where we take $E$ instead of $S$. Here, we denote this gauge by $\varphi$. Let $\theta$ be an $L$-linear anisotropic involution on $C$ such that $\varphi$ is invariant under $\theta$. Let us denote by $\theta_0$ the induced involution on $\gr_\varphi (C)_0=E$. We have the following:

\begin{lemma}\label{lm2.2}
The algebras with involution $(C,\theta)$ and $(E, \theta_0)$ are as above. Denote by $\gamma$ the canonical involution  on $Q$ and let  $1,i,j,ij$ be a quaternion base of $Q$ such that $i^2=t_1$ and $j^2=t_2$. We have the decomposition
\[
(C,\theta)\simeq (E, \theta_0)\otimes_F(Q, \Int(u)\circ\gamma),
\]
where $u$ is one of the elements $1,i,j,ij$. 
\end{lemma}
\begin{proof}
As in Example \ref{example}, recall that
\[
\gr_\varphi(C)=E\otimes_F(\widetilde{t_1}, \widetilde{t_2})_{\gr_v(L)}
\]
where $v$ denotes the $(t_1, t_2)$-adic valuation on $L$. Since $\widetilde\theta$ induces $\theta_0$ on $E$, it suffices to evaluate the restriction $\widetilde\theta\vert_{\gr_\varphi(Q)}$ of $\widetilde\theta$ to $\gr_\varphi(Q)=(\widetilde{t_1}, \widetilde{t_2})_{\gr_v(L)}$. Notice that the degrees $(\tfrac{1}{2},0)$ and $(0,\tfrac{1}{2})$ components of $\gr_\varphi(C)$ are respectively $E\otimes \widetilde i$ and $E\otimes \widetilde j$. Since $\theta$ preserves the filtration on $\gr_\varphi(C)$ defined by the gauge $\varphi$, one has $\widetilde\theta(1\otimes \widetilde i)= \zeta_1\otimes \widetilde i$ for some $\zeta_1\in E$. Moreover, it is easy to see that $\zeta_1\otimes \widetilde i$ centralizes  $E$ since $1\otimes \widetilde i$ centralizes $E\otimes 1$ and $\widetilde\theta(1\otimes \widetilde i)=\zeta_1\otimes \widetilde i$. Therefore $\zeta_1\in F$. On the other hand, we have
\[
1\otimes \widetilde i= {\widetilde\theta}^2(1\otimes \widetilde i)= \widetilde\theta(\zeta_1\otimes \widetilde i)= (\theta_0(\zeta_1)\zeta_1)\otimes \widetilde i.
\]
This shows that $\theta_0(\zeta_1)\zeta_1=1$, and so $\zeta_1=\pm 1$. Likewise, $\widetilde\theta(1\otimes \widetilde j)=\zeta_2\otimes \widetilde j$ with $\zeta_2=\pm 1$. Thus, we have:
\[
\widetilde\theta\vert_{\gr_\varphi(Q)}=\left\{
\begin{array}{ll}
\gamma & \textrm{if $\zeta_1=\zeta_2=-1$},\\
\Int(\widetilde j)\circ \gamma & \textrm{if $\zeta_1=1, \zeta_2=-1$},\\
\Int(\widetilde i)\circ \gamma & \textrm{if $\zeta_1=-1, \zeta_2=1$},\\
\Int(\widetilde {ij})\circ \gamma & \textrm{if $\zeta_1=1, \zeta_2=-1$}.\\
\end{array}
\right.
\]
Hence, our arguments show that $\widetilde\theta = \theta_0\otimes(\Int(\widetilde u)\circ\gamma)$ where $u$ is one of the quaternion base elements $1,i,j,ij$ of $Q$.  It then follows by ~\cite[Thm. 2.6]{TW11} that the algebras with involution $(E\otimes_FQ,\theta)$ and $(E\otimes_FQ, \theta_0\otimes\Int(u)\circ\gamma)$ are isomorphic. This concludes the proof.
\end{proof}
From now on, we assume that  $(C,\theta)$ is totally decomposable. We want to show that  $(E, \theta_0)$ is totally decomposable. We first introduce the notion of an armature, which  originated in  ~\cite{Tig82} and ~\cite{TW87}. However, we shall refer frequently to the recent book ~\cite{TW14} that gives a more extensive treatment of armatures.   Let $A$ be a finite-dimensional  algebra over a field $F$. Recall from ~\cite[Def. 7.27]{TW14} that a subgroup $\cA\subset A\m/F\m$ is an \emph{armature} of $A$ if $\cA$ is abelian, $A$ has cardinality  $\vert \cA\vert=[A:F]$, and $\{a\in A \, \vert\, aF\m\in \cA\}$ spans $A$ as an $F$-vector space. For $a\in A^{\times}/F^{\times}$, we fix a representative  $x_a$ in $A$ whose image in $A^{\times}/F^{\times}$ is $a$, that is, $a=x_aF^{\times}$. In what follows, we always suppose $A$ is a central simple algebra of exponent $2$ over $F$, so the exponent of the group $\cA$ is $2$, that is, $x_a^2\in F\m$ for all $a\in \cA$.

For example, let $A$ be a quaternion algebra over $F$. The image in $A\m/F\m$ of the standard generators $i,j$ generate an armature of $A$.  More generally, in a tensor product of quaternion algebras the images of the product of standard generators generate an armature. Actually, there exists a characterization of tensor products of quaternion algebras in terms of the existence of armatures: a central simple algebra $A$ over $F$ of exponent $2$  has an armature if and only if $A$ is isomorphic to a tensor product of quaternion algebras over $F$ (see ~\cite[Cor. 7.34]{TW14}).

For an armature $\cA$ of a finite-dimensional $F$-algebra $A$ of exponent $2$,  there is an associated \emph{armature pairing} 
\[
\langle\,,\,\rangle : \cA\times\cA \to \mu_2(F)=\{\pm 1\} \quad \text{defined by}\quad  \langle a,b\rangle=x_ax_bx^{-1}_ax_b^{-1}.
\]
It is shown in ~\cite[Prop. 7.26]{TW14} that $\langle\,,\,\rangle$ is a well-defined alternating bimultiplicative pairing. The set $\{a_1,\ldots,a_r\}$ is called a base of $\cA$ if $\cA$ is  the internal direct product 
\[ 
\cA=(a_1)\times\cdots\times (a_r)
\]
where $(a_i)$ denotes  the cyclic subgroup of $\cA$ generated by $a_i$. If $\langle\,,\,\rangle$ is nondegenerate then $(\cA,\langle\,,\,\rangle)$ is a symplectic module and has a \emph{symplectic base} with respect to $\langle\,,\,\rangle$, i.e, a base $\{a_1,b_1,\ldots, a_n,b_n\}$ such that for all $i,j$
\[
\langle a_i,b_i\rangle= -1 ,\text{ where } \ord(a_i)=\ord(b_i)=2
\]
\[
\langle a_i,a_j\rangle=\langle b_i,b_j\rangle=1 \text{ and, if } i\ne j,\,\, \langle a_i,b_j\rangle=1
\]
(see ~\cite[Thm. 7.2]{TW14} or ~\cite[1.8]{Tig82}). Given any subgroup $\cB$ of $\cA$, let $F[\cB]$ denote the $F$-subspace (and subalgebra) of $A$ generated by the representatives of the elements of $\cB$. It is easy to check  that $\cB$ is an armature of $F[\cB]$.

Let $v$ be a valuation on $F$. Assume that $A$ has an armature $\cA$. Using the armature $\cA$,  an $F$-gauge on $A$ is constructed in ~\cite[\S 7.2.3]{TW14}:  a map $g_\cA\,: A\, \longrightarrow\, \Gamma$ is defined as follows. We choose an $F$-basis $(x_a)_{a\in \cA}$ of $A$ and set 
\[
g_\cA(\sum_{a\in \cA}\lambda_ax_a)=\min_{a\in \cA}(v(\lambda_a)+\tfrac{1}{2}v(x_a^2))\in \Gamma\cup\{\infty\} \text{ for } \lambda_a\in F,
\]
where we recall that $x_a^2\in F\m$ for all $a\in \cA$. One also defines a map $\overline{g}_\cA\,: \cA\longrightarrow \Gamma/\Gamma_F$ by
\[
\overline{g}_\cA(a)= g_\cA(x_a)+\Gamma_F.
\]
The map $\overline{g}_\cA$ is well-defined since $x_a$ is uniquely determined by $a$ up to a factor in $F$. The map $g_\cA$ is a $v$-gauge on $A$ and $\overline{g}_\cA$ is a surjective group homomorphism, see ~\cite[Lemma 7.46, Cor. 7.48 and Thm. 7.49]{TW14}. Following Tignol-Wadsworth's terminology in ~\cite{TW14}, the gauge $g_\cA$ is called the \emph{armature gauge} associated to $\cA$.

We now return to the central simple $L$-algebra with involution $(C,\theta)$ defined above. We fix the setting we will consider in the next proposition. Assume that $C$ has an armature $\cC$ and $\theta(x_a)=\pm x_a$ for all $a\in \cC$.  Let $g_\cC$ be the armature gauge associated to $\cC$. It is  clear from the definition of the armature gauge $g_\cC$ that $g_\cC(\theta(c))=g_\cA(c)$ and $g_\cC(\theta(c)c)=2g_\cC(c)$ for all $c\in C$. The $(t_1, t_2)$-adic valuation  of $L$ being Henselian, and $\theta$ being anisotropic, $g_\cC$ is  the unique  gauge on $C$ invariant under $\theta$ (see ~\cite[Thm. 2.2]{TW11}), that is,  $g_\cC=\varphi$ (where $\varphi$ is the gauge used in the proof of Lemma \ref{lm2.2}) . Set $\cC_0=\ker \overline{g}_\cC$. Since $\Gamma_C/\Gamma_F=(\tfrac{1}{2}\Z/\Z)^2$, it follows from the surjectivity of $\overline{g}_\cC$ that
\[
|\cC_0|=|\cC|/|(\tfrac{1}{2}\Z/\Z)^2|=\tfrac{1}{4}|\cC|=\dim_F E.
\]
We have the following result.
\begin{proposition}\label{prop2.1}
The algebras with involution $(C,\theta)$ and $(E, \theta_0)$ are as in Lemma \ref{lm2.2}.  Assume that $(C,\theta)$ is totally decomposable. Then $(E, \theta_0)$ is totally decomposable.
\end{proposition}
\begin{proof}
Write $(C,\theta)=\bigotimes_{m=1}^{k}(Q_m, \theta_m)$ for some quaternion algebras $Q_m$ and some involutions $\theta_m$ on $Q_m$  (for $m=1,\ldots, k$). We choose the quaternion generators $i_m$ and $j_m$ of $Q_m$ such a way that $\theta(i_m)=\pm i_m$ and $\theta(j_m)=\pm j_m$. Let $\cC$ be the armature of $C$ associated to the generators $i_m, j_m$ for $m=1,\ldots, k$, that is, $\cC$ is the subgroup of $C\m/L\m$ generated by the image of $\{i_1, j_1, \ldots, i_k, j_k\}$. It is shown in ~\cite[Prop. 4.8]{TW10} that the degree $0$ component of $\gr_\varphi(C)$, which is $\gr_\varphi(C)_0=E$,  has an armature canonically isomorphic to $\cC_0$ with armature  pairing isometric to the restriction to $\cC_0$ of the armature pairing of $\cC$. We recall that the gauge $\varphi$ is nothing but the armature gauge associated to $\cC$ by uniqueness arguments. Since the center $F$ of $E$ is a field, ~\cite[Cor. 2.8]{TW87} indicates that $E$ is a tensor product of quaternion algebras. In other words, let $\{a_1, b_1,\ldots, a_{k-1}, b_{k-1} \}$ be a symplectic base of $\cC_0$ and let denote by $F[a_i, b_i]$ the quaternion subalgebra of $E$ generated by the representatives $x_{a_i}, x_{b_i}$ of $a_i, b_i$. The algebra $E$ is the tensor product $E=F[a_1, b_1]\otimes \cdots \otimes F[a_{k-1}, b_{k-1}]$.  Since $\theta(x_a)=\pm x_a$ for all $a\in \cC$,  the algebra with involution $(E,\theta_0)$ is totally decomposable. The proof is complete.
\end{proof}
\begin{lemma}\label{lm2.3}
Let $\theta_1$ and $\theta_2$ be $L$-linear anisotropic involutions on $M_2(L)$ and $Q$ respectively. Then there exist $\lambda_0\in F^\times$  such that
\[
(M_2(L), \theta_1) \otimes (Q,\theta_2)\simeq (M_2(F), \ad_{\langle\langle \lambda_0 \rangle\rangle})\otimes (Q,\theta_2).
\]
\end{lemma}
\begin{proof}
Setting $V=Q\times Q$, there is $\lambda\in L\m$ such that $\theta_1\otimes \theta_2$ is the adjoint of the binary Hermitian form $h=\langle 1, -\lambda\rangle: V\times V \to Q$ with respect to  $\theta_2$ on $Q$ (see for instance ~\cite[Thm. (4.2)]{KMRT98}). We claim that we may always suppose $w(\lambda)=(0,0)$. Indeed, multiplying $\lambda$ by a square if it is necessary, the valuation  $w(\lambda)$ is either $(0,0)$ or $(1,0)$ or $(0, 1)$ or $(1,1)$. For instance, assume $w(\lambda)=(1, 0)$. We set $\theta'=\theta_1\otimes \theta_2$ for simplicity.  Choose a quaternion base $1, i, j, ij$ of $Q$ such a way that $\theta'(1\otimes i)=\pm 1\otimes i$, $\theta'(1\otimes j)=\pm 1\otimes j$ and $i^2=t_1$. Let $e_1, e_2$ be the base of $V$ corresponding to the diagonalization $h=\langle 1, -\lambda \rangle$.  Replacing the base $e_1, e_2$ by $e_1, i^{-1}e_2$, we get the isometry $h=\langle 1, -\lambda \rangle\simeq\langle 1, -\theta' (1\otimes i)^{-1}\lambda(1\otimes i)^{-1} \rangle$. Moreover, we have  $w(\theta' (1\otimes i)^{-1}\lambda(1\otimes i)^{-1})=0$. Hence, we may always suppose $w(\lambda)=(0,0)$ as claimed. 

Since $w(\lambda)=(0,0)$, one has $\bar\lambda=\lambda_0$ for some $\lambda_0\in F\m$. Thus, $\overline{\lambda\lambda_0\inv}=1$.  By Hensel's lemma it follows that there exists $l\in L\m$ such that $\lambda_0l^2=\lambda$. Hence, $\langle 1, -\lambda \rangle= \langle 1, -\lambda_0l^2 \rangle\simeq \langle 1, -\lambda_0 \rangle$ for some $\lambda_0\in F\m$. So
\[
(M_2(L)\otimes Q, \theta')\simeq (M_2(F), \ad_{\langle\langle \lambda_0\rangle\rangle})\otimes (Q, \theta_2).
\]
The lemma is proved.
\end{proof}

\medskip
\begin{proof}[Proof of Theorem \ref{thm2.1}]
Assume that $(S_2,\tau)$ is totally decomposable with $n$ split factors, say
\begin{equation*}
(S_2,\tau)\simeq \bigotimes_{k=1}^{n}(M_2(L),\tau_k)\otimes(B,\eta),
\end{equation*}
for some involutions $\tau_k$ on $M_2(L)$ and some totally decomposable involution $\eta$ on the centralizer $B$ of $M_{2^n}(L)$ in $S_2$. 

We first show that  $B=B_0\otimes Q$ for some subalgebra $B_0$ of $B$ defined over $F$.  For this, let $D$ be the division algebra Brauer equivalent  to $S$. The algebra $S_2=S\otimes Q$ is Brauer equivalent to $B\sim D\otimes Q$. Hence $B\simeq M_r(D)\otimes Q$ for some $r\in \N$. It then follows that $B\simeq B_0\otimes Q$ with $B_0\simeq M_r(D)$ defined over $F$.

 As in Lemma \ref{lm2.2}, notice that there exists  a $\eta$-invariant gauge on $B$ and 
\[
(B,\eta)\simeq (B_0,\eta_0)\otimes (Q,\Int(u)\circ\gamma)
\]
for some $u\in Q\m$, where $\eta_0$ is the induced involution on $B_0$ by this gauge and $\gamma$ is the canonical involution on $Q$. Thus, we have
\[
(S_2,\tau)\simeq \bigotimes_{k=1}^{n}(M_2(L),\tau_k)\otimes(B_0,\eta_0)\otimes (Q,\Int(u)\circ\gamma).
\]
Moreover the involution $\eta_0$ is totally decomposable by Prop.~\ref{prop2.1} since $\eta$ is totally decomposable. Applying repeatedly  Lemma \ref{lm2.3},  one finds  $\lambda_1,\ldots \lambda_n\in F\m$ such that
\[
(S_2,\tau)\simeq \bigotimes_{k=1}^{n}(M_2(F),\ad_{\langle\langle \lambda_k \rangle\rangle})\otimes(B_0,\eta_0)\otimes (Q,\Int(u)\circ\gamma).
\]
Consider the gauge $g_2$ on $S_2$ constructed in Example \ref{example}. The residue algebra with involution $\gr_{g_2}(S_2, \tau)_0$  is  
\[
\gr_{g_2}(S_2, \tau)_0\simeq (S,\rho)\simeq  \bigotimes_{k=1}^{n}(M_2(F),\ad_{\langle\langle \lambda_k \rangle\rangle})\otimes(B_0,\eta_0).
\]
Since $(B_0,\eta_0)$ is totally decomposable,  $(S,\rho)$ is totally decomposable with $n$ split factors. That concludes the proof.
\end{proof}
Now, let $(A',\sigma')$ be the algebra with orthogonal involution of Prop.~\ref{sivatski}. Notice that $A'\otimes Q$ is of degree $16$ and index $8$, and it is totally decomposable with $1$ split factor since $A'$ is totally decomposable with $1$ split factor. It is also clear that $(A',\sigma')\otimes (Q,\rho)$ is totally decomposable. But, we have the following:
\begin{corollary}\label{cor5.1}
The algebra with involution $(A',\sigma')\otimes (Q,\rho)$ is not totally decomposable with $1$ split factor.
\end{corollary}
\begin{proof}
Assume that $(A',\sigma')\otimes (Q,\rho)$ is totally decomposable with $1$ split factor. It follows from Thm.~\ref{thm2.1} that $(A',\sigma')$ is totally decomposable with $1$ split factor; that is impossible. Therefore $(A',\sigma')\otimes (Q,\rho)$ is not totally decomposable with $1$ split factor.
\end{proof}
\begin{remark}\label{rem5.1}
 Combining Corollaries \ref{cor4.1} and \ref{cor5.1}, one can construct  examples of  totally decomposable algebras with involution of index $2^r$, where $r\geq 2$ for orthogonal involutions and $r\geq 3$ for symplectic involutions, and co-index $2^k$ with $k\geq 1$ that are not  totally decomposable with $k$ split factors.
\end{remark}
\subsection*{Acknowledgements}
I am very grateful to Jean-Pierre Tignol for triggering the idea presented in this paper and for his availability. I would like to thank  Anne Qu\'eguiner-Mathieu for inspiring discussion and for the financial support provided by the French Agence Nationale de la Recherche (ANR) under reference ANR-12-BL01-0005. I further gratefully acknowledge the financial support provided by the French Embassy in Bamako  (Grant N$^\text{o}$ 830592G (2014)).

\end{document}